\documentclass[a4paper,reqno]{amsart}
\usepackage{amssymb, amsmath, amscd}
\usepackage[final]{graphicx}
\usepackage{float}\usepackage{wrapfig}
\usepackage{color}
\usepackage{bbm}
\def\MM{\mathcal{M}}

\def\pone{\mathbbm{1}}
\newtheorem{theorem}{Theorem}[section]

\newtheorem{lemma}[theorem]{Lemma}
\newtheorem{remark}[theorem]{Remark}

\makeatletter
 \@addtoreset{equation}{section}
\makeatother
\begin{document}
\title{Spaces of locally homogeneous affine surfaces}
%\author{E. Puffini}
%\address{Krill Institute of Technology, Islas Malvinas}
%\email{ekaterinapuffin@yahoo.com}
\author{M. Brozos-V\'{a}zquez \, E. Garc\'{i}a-R\'{i}o\, P. Gilkey}
\address{MBV: Universidade da Coru\~na, Differential Geometry and its Applications Research Group, Escola Polit\'ecnica Superior, 15403 Ferrol,  Spain}
\email{miguel.brozos.vazquez@udc.gal}
\address{EGR: Faculty of Mathematics, University of Santiago de Compostela, 15782 Santiago de Compostela, Spain}
\email{eduardo.garcia.rio@usc.es}
\address{PBG: Mathematics Department, University of Oregon, Eugene OR 97403-1222, USA}
\email{gilkey@uoregon.edu}
\thanks{Supported by project MTM2016-75897-P (Spain).}
\begin{abstract}
We examine the topology of various spaces of locally homogeneous affine manifolds which arise from the classification result of Opozda \cite{Op04} as orbits of the action of $GL(2,\mathbb{R})$ (Type~$\mathcal{A}$) and the $ax+b$ group (Type~$\mathcal{B}$).
We determine the topology of the spaces of Type~$\mathcal{A}$ models in relation to
the rank of the Ricci tensor. We determine the topology of the spaces of
Type~$\mathcal{B}$ models which either are flat or where the Ricci tensor is alternating.
\end{abstract}
\keywords{Homogeneous affine surface, linear equivalence, Ricci tensor}
\subjclass[2010]{53A15, 53C05, 53B05}
\maketitle

\section{Introduction}

\subsection{Notational conventions}
An {\it affine surface} is a pair $\mathcal{M}=(M,\nabla)$ where $M$ is a smooth surface and where $\nabla$
is a torsion free connection on the tangent bundle of $M$. 
Let $x=(x^1,x^2)$ be a system of local coordinates on $M$.
Adopt the {\it Einstein convention} and sum over repeated indices to express 
$\nabla_{\partial_{x^i}}\partial_{x^j}=\Gamma_{ij}{}^k\partial_{x^k}$. The {\it Christoffel symbols} $\Gamma=\{\Gamma_{ij}{}^k\}$
determine the connection in the coordinate chart. Let $\rho$ be the associated Ricci tensor.
The Ricci tensor carries the geometry in dimension $2$; an affine surface is flat if and only if $\rho=0$.
Since the Ricci tensor of an affine manifold is not necessarily symmetric, let $\rho_s(X,Y)=\frac{1}{2}\{\rho(X,Y)+\rho(Y,X)\}$ and  $\rho_a(X,Y)=\frac{1}{2}\{\rho(X,Y)-\rho(Y,X)\}$ be the \emph{symmetric} and \emph{alternating} Ricci tensors.

\subsection{Locally homogeneous affine surface geometries}
Work of Opozda \cite{Op04} shows that any locally homogeneous affine surface $\mathcal{M}$ 
is modeled on one of the following geometries.
\begin{itemize}
\item {\bf \underline{Type~$\mathcal{A}$}.} $\mathcal{M}=(\mathbb{R}^2,\nabla)$
with constant Christoffel symbols $\Gamma_{ij}{}^k=\Gamma_{ji}{}^k$. 
This geometry is homogeneous; the Type~$\mathcal{A}$ connections are the
left invariant connections on the Lie group $\mathbb{R}^2$.
\item {\bf \underline{Type~$\mathcal{B}$}.} $\mathcal{M}=(\mathbb{R}^+\times\mathbb{R},\nabla)$ with Christoffel symbols 
$\Gamma_{ij}{}^k=(x^1)^{-1}A_{ij}{}^k$ where $A_{ij}{}^k=A_{ji}{}^k$ is constant. This geometry is homogeneous;
the Type~$\mathcal{B}$ connections are the left invariant connections on the $ax+b$ group.
\item {\bf \underline{Type~$\mathcal{C}$}.} $\mathcal{M}=(M,\nabla)$ where $\nabla$ is the Levi-Civita connection of
the round sphere $S^2$.
\end{itemize}

This result has been applied by many authors. Kowalski and Sekizawa~\cite{KS14} used it to examine Riemannian extensions of 
affine surfaces, Vanzurova~\cite{V13} used it to study the metrizability of locally homogeneous affine surfaces, and  
D\v usek~\cite{Du10} used it to
study homogeneous geodesics. 
It plays a central role in the study of locally homogeneous connections with torsion of 
Arias-Marco and Kowalski~\cite{AMK08} (see also \cite{BVGRG2019} for a unified treatment independently of the torsion tensor). 
Although we will work with the local theory, the compact setting has been examined in \cite{AG14,Opozda}.

The Ricci tensor $\rho$ of an affine surface determines the full curvature tensor. 
In Section~\ref{section-type-A}, we examine the spaces where the Ricci tensor has fixed rank
in the Type~$\mathcal{A}$ setting. In
Section~\ref{section-type-B}, we consider the spaces where either the Ricci tensor vanishes identically
or where the Ricci tensor is alternating and non-trivial in the Type~$\mathcal{B}$ setting.

\subsection{Type~$\mathcal{A}$ geometries}\label{S1.3}
Let $\mathcal{M}(a,b,c,d,e,f):=(\mathbb{R}^2,\nabla)$ where the Christoffel symbols of $\nabla$ are constant and given
by
\begin{equation}\label{E1.a}\begin{array}{llllll}
\Gamma_{11}{}^1=a,&\Gamma_{11}{}^2=b,&\Gamma_{12}{}^1=\Gamma_{21}{}^1=c,\\
\Gamma_{12}{}^2=\Gamma_{21}{}^2=d,&\Gamma_{22}{}^1=e,&\Gamma_{22}{}^2=f\,.
\end{array}\end{equation}
This identifies the set of Type~$\mathcal{A}$ geometries with $\mathbb{R}^6$. 
The linear transformations 
$T(x^1,x^2)=(a^1_1 x^1 + a^1_2 x^2, a^2_1 x^1+a^2_2 x^2)$ where 
$(a_i^j)\in\operatorname{GL}(2,\mathbb{R})$ act on the set of Type~$\mathcal{A}$ geometries.
We say that two Type~$\mathcal{A}$ surface models are \emph{linearly equivalent} if
there exists $T\in\operatorname{GL}(2,\mathbb{R})$  intertwining the two structures.
One has that two Type~$\mathcal{A}$  surfaces  with non-degenerate Ricci tensor are affine equivalent if and only if they are linearly equivalent
(see \cite{BGG18}). On the contrary, there  exist Type~$\mathcal{A}$ surfaces with degenerate Ricci tensor 
which are not linearly equivalent
but which nevertheless are affine equivalent. We refer to the discussion in \cite{GV18} for further details.

 We consider the induced action of $GL(2,R)$ on $\mathbb{R}^6$ 
and identify the linear orbit of a Type~$\mathcal{A}$ model $\mathcal{M}$ with 
$\mathcal{S}(\mathcal{M})=GL(2,\mathbb{R})/\mathcal{I}(\mathcal{M})$ where $\mathcal{I}(\mathcal{M})$ is the isotropy group
$\mathcal{I}(M)=\{ T\in GL(2,\mathbb{R}); T^*\mathcal{M}=\mathcal{M}\}$.

It was shown in  \cite{GV18} that any  flat  Type~$\mathcal{A}$ model is linearly equivalent to one
of the following:
 \begin{equation}
 \label{eq:models0}
  \begin{array}{ll}
 \MM_0^0:=\mathcal{M}(0,0,0,0,0,0),&
 \MM_1^0:=\mathcal{M}(1,0,0,1,0,0),\\
 \noalign{\medskip}
 \MM_2^0:=\mathcal{M}(-1,0,0,0,0,1), &
 \MM_3^0:=\mathcal{M}(0,0,0,0,0,1),\\
 \noalign{\medskip}
 \MM_4^0:=\mathcal{M}(0,0,0,0,1,0), &
 \MM_5^0:=\mathcal{M}(1,0,0,1,-1,0).
 \end{array}
 \end{equation}
The structure $\MM_0^0$ is a singular cone point. The next result shows the remaining orbits $\mathcal{S}(\mathcal{M}_i^0):=\operatorname{GL}(2,\mathbb{R})\cdot\mathcal{M}_i^0$
for $1\le i\le 5$ glue together to define
a smooth 4-dimensional submanifold of $\mathbb{R}^6$.
Let $\pone$ be the trivial line bundle
over the circle $S^1$, let $\mathbb{L}$ be the M\"obius line bundle over $S^1$, and
let $\mathcal{A}^0\subset\mathbb{R}^6\setminus\{0\}$ be the set of all flat Type~$\mathcal{A}$ geometries other than the cone point $\mathcal{M}_0^0$.

\begin{theorem}\label{T1.1}
$\mathcal{A}^0$ is a smooth submanifold of $\mathbb{R}^6\setminus\{0\}$ diffeomorphic to the total space of
$\mathbb{L}\oplus\pone\oplus\pone$ 
minus the zero section.
\end{theorem}

The Ricci tensor of any Type~$\mathcal{A}$ model is symmetric. 
Let $\mathcal{A}^{1}_{\pm}\subset\mathbb{R}^6$ be the set of all Type~$\mathcal{A}$ geometries where the Ricci tensor has rank 1 and is
positive semi-definite ($+$) or negative semi-definite $(-)$.
Any element in $\mathcal{A}^{1}_{\pm}$ is linearly equivalent to one of the following, where 
$c\in \mathbb{R}$ and $c_1\in \mathbb{R}\setminus \{0,-1\}$ (see \cite{BGG18, GV18}):
\begin{equation}\label{eq:models1}
 \begin{array}{ll}
 \MM_1^1:=\mathcal{M}(-1,0,1,0,0,2),\\[0.03in]
 \MM_2^1(c_1):=\mathcal{M}(-1,0,c_1,0,0,1+2c_1),\\[0.03in]
  \MM_3^1(c_1):=\mathcal{M}(0,0,c_1,0,0,1+2c_1),\\[0.03in]
 \MM_4^1(c):=\mathcal{M}(0,0,1,0,c,2),\\[0.03in]
  \MM_5^1(c):=\mathcal{M}(1,0,0,0,1+c^2,2c).
 \end{array}
\end{equation}
We will see in Lemma~\ref{L4.2} that the orbit structure of the action of 
$\operatorname{GL}(2,\mathbb{R})$ on $\mathcal{A}^1_{\pm}$
is quite complicated. It is therefore, perhaps, a bit surprising that the set of all orbits $\cup_{i,c}\mathcal{M}_i^1(c)\cdot\operatorname{GL}(2,\mathbb{R})$
is smooth as shown in the following result. 

%While the orbit of $ \MM_1^1$ is $4$-dimensional and the other orbits are $5$-dimensional, all of them glue smoothly into a $5$-dimensional manifold.
\begin{theorem}\label{T1.2}
 $\mathcal{A}^{1}_{\pm}$ is a smooth submanifold of $\mathbb{R}^6$ diffeomorphic to
$S^1\times S^1\times\mathbb{R}^3$.
\end{theorem}

 The remaining geometries where the Ricci tensor has rank 2
form an open subset $\mathbb{R}^6\setminus\{\{0\}\cup\mathcal{A}^0\cup\mathcal{A}^1_+\cup\mathcal{A}^{1}_-\}$.

These results should be contrasted with the results in \cite{kv03} where it is shown that any Type~$\mathcal{A}$ affine surface is linearly equivalent to a surface determined by at most two non-zero parameters.

\subsection{Type~$\mathcal{B}$ geometries}\label{S1.4}
Let 
$\mathcal{N}(a,b,c,d,e,f):=(\mathbb{R}^+\times\mathbb{R},\nabla)$ where the Christoffel symbols of $\nabla$ are given by
\begin{equation}\label{E1.c}\begin{array}{llllll}
\displaystyle\Gamma_{11}{}^1=\frac a{x^1},&\displaystyle\Gamma_{11}{}^2=\frac b{x^1},&
\displaystyle\Gamma_{12}{}^1=\Gamma_{21}{}^1=\frac c{x^1},\\
\Gamma_{12}{}^2=\Gamma_{21}{}^2=\displaystyle\frac d{x^1},&\Gamma_{22}{}^1=\displaystyle\frac e{x^1},&
\displaystyle\Gamma_{22}{}^2=\frac f{x^1}\,.
\end{array}\end{equation}
This identifies the space of Type~$\mathcal{B}$ geometries with $\mathbb{R}^6$.

The natural structure group here is not the full general linear group, but rather the $ax+b$ group.
We let $T_{a,b}(x^1,x^2):=(x^1,ax^2+bx^1)$ define an action of the $ax+b$ group on $\mathbb{R}^+\times\mathbb{R}$;
this acts on the Type~$\mathcal{B}$ geometries by reparametrization and defines the natural notion of linear equivalence
in this setting. 
%Let $\mathcal{N}=(\mathbb{R}^+\times\mathbb{R},\nabla)$ be a Type~$\mathcal{B}$ affine surface model. The $a\vec x+b$ group preserves this geometry and 
Thus, two Type~$\mathcal{B}$ models $\mathcal{N}_1$ and $\mathcal{N}_2$ are said to be \emph{linearly equivalent} if and only if there exists  an affine transformation of the form
$\Psi(x^1,x^2)=(x^1, a^2_1 x^1+ a^2_2 x^2)$ for $ a^2_2\ne0$ intertwining the two structures.
It follows from the work in \cite{BGG18, BGG16} that two Type~$\mathcal{B}$ surfaces which are { neither flat
nor} of Type~$\mathcal{A}$
are affine isomorphic if and only if they are linearly isomorphic.
This is a non-trivial observation as there are non-linear affine transformations from one model to another if the dimension of the space of affine Killing vector fields is $4$-dimensional { or if the geometry is flat and thus the dimension of the space
of affine Killing vector fields is $6$-dimensional}.

It was shown in \cite{GV18a} that a flat Type~$\mathcal{B}$ model is linearly equivalent to one of the following models:
 $$
 \begin{array}{ll}
 \mathcal{N}^0_0:=\mathcal{N}(0,0,0,0,0,0), &\mathcal{N}^0_1(\pm):=\mathcal{N}(1,0,0,0,\pm1,0),
 \\
 \noalign{\medskip}
 \mathcal{N}^0_2(c_1):=\mathcal{N}(c_1-1,0,0,c_1,0,0), \, c_1\ne0, &\mathcal{N}^0_3:=\mathcal{N}(-2,1,0,-1,0,0),
 \\
 \noalign{\medskip}
 \mathcal{N}^0_4:=\mathcal{N}(0,1,0,0,0,0), &\mathcal{N}^0_5:=\mathcal{N}(-1,0,0,0,0,0),
 \\
 \noalign{\medskip}
 \mathcal{N}^0_6(c_2):=\mathcal{N}(c_2,0,0,0,0,0),\, c_2\ne0,-1.&
 \end{array}
 $$

Let $\mathcal{B}^0\subset\mathbb{R}^6$ be the space of flat Type~$\mathcal{B}$ geometries other
than the cone point $\mathcal{N}_0^0$ determined by the origin in $\mathbb{R}^6$.
Unlike the Type~$\mathcal{A}$ setting described in Theorem~\ref{T1.1},
$\mathcal{B}^0$ is not a smooth manifold but consists of the union of 3 smooth submanifolds of $\mathbb{R}^6$
which intersect transversally along the union of 3 smooth curves in $\mathbb{R}^6$.
Define
\begin{equation}\label{eq:paramet1}
\begin{array}{ll}
\mathcal{U}_1(r,s):=\mathcal{N}(1+rs^2,-s(1+rs^2),rs,-rs^2,r,-rs),&\mathcal{B}_1:=\operatorname{Range}\{\mathcal{U}_1\},\\[0.05in]
\mathcal{U}_2(u,v):=\mathcal{N}(u,v,0,0,0,0),&\mathcal{B}_2:=\operatorname{Range}\{\mathcal{U}_2\},\\[0.05in]
\mathcal{U}_3(u,v):=\mathcal{N}(u,v,0,1+u,0,0),&\mathcal{B}_3:=\operatorname{Range}\{\mathcal{U}_3\}.
\end{array}
\end{equation}
\medbreak\noindent 

\begin{theorem}\label{T1.5}
$\mathcal{B}^0
=\mathcal{B}_1\cup\mathcal{B}_2\cup\mathcal{B}_3$.
$\mathcal{B}_2$ and $\mathcal{B}_3$ are closed smooth surfaces
in $\mathbb{R}^6$ which are diffeomorphic to $\mathbb{R}^2$ and which intersect transversally along the
curve $\mathcal{N}(-1,v,0,0,0,0)$ for $v\in\mathbb{R}$. 
$\mathcal{B}_1$ can be completed to a smooth closed surface $\tilde{\mathcal{B}}_1$
which intersects $\mathcal{B}_2$ transversally along the curve
$\mathcal{N}(1,v,0,0,0,0)$ and which intersects $\mathcal{B}_3$
transversally along the curve $\mathcal{N}(0,v,0,1,0,0)$ for $v\in\mathbb{R}$.
\end{theorem}

In the Type~$\mathcal{B}$ setting, it is possible for the symmetric Ricci tensor 
$\rho_s$ to vanish without the geometry being flat;
this is not possible in the Type~$\mathcal{A}$ setting. The alternating Ricci tensor, $\rho_a$, carries the geometry
in this context. 

%{\red It was shown in \cite{BGG18} that any Type~$\mathcal{B}$ model with alternating Ricci tensor is linearly equivalent to one of the following models:
%$$
%\begin{array}{ll}
%\mathcal{N}_1(c):=\mathcal{N}(0,c,1,0,0,1), & c\in\mathbb{R},
%\\
%\noalign{\medskip}
%\mathcal{N}_2(c,\pm):=\mathcal{N}(1\mp c^2),c,0,\mp c^2,\pm 1,\pm 2c), & c>0.
%\end{array}
%$$
%}

Let $\mathcal{B}_a$ be the set of all
 Type~$\mathcal{B}$ structures where $\rho_s=0$ but $\rho_a\ne0$. Set
\begin{equation}
\label{eq:paramet2}
\begin{array}{ll}
\mathcal{V}_1(r,s,t):=\mathcal{N}(s, t,  r, 0, 0, r),\\[0.05in]
\mathcal{V}_2(u,v,w):=\mathcal{N}( 1 - 2 u w + v w^2,  w (1 - u w + v w^2),   u - v w,  -v w^2, v,  u + v w)
\end{array}
\end{equation}
and let $\mathcal{D}_1:= \operatorname{Range}\{\mathcal{V}_1\}$ and $\mathcal{D}_2:= \operatorname{Range}\{\mathcal{V}_2\}$.
\begin{theorem}\label{T1.7} $\mathcal{B}_a=\mathcal{D}_1\cup \mathcal{D}_2$. $\mathcal{V}_i$ defines smoothly embedded 
 3-dimensional submanifolds of $\mathbb{R}^6$ for
 $r\ne0$ and $u\ne0$ which intersect transversally along a smooth 2-dimensional submanifold. 
\end{theorem}

\section{The space of Type~$\mathcal{A}$ models}\label{section-type-A}
Let $\mathcal{M}(a,b,c,d,e,f):=(\mathbb{R}^2,\nabla)$ { be given
by Equaton~\eqref{E1.a} where the parameters $(a,b,c,d,e,f)$ are real constants}.
The { associated} Ricci tensor is symmetric.

\subsection{The space of flat Type~$\mathcal{A}$ models}
Since the Ricci tensor determines the curvature in dimension two, flat surfaces are determined by a vanishing Ricci tensor. We provide the proof of the first result of the paper as follows.

%We say that two
%Type~$\mathcal{A}$ structures are {\it linearly equivalent} if they 
%differ by an element of $\operatorname{Gl}(2,\mathbb{R})$.
%{\red Affine equivalence reduces to linear equivalence whenever the Ricci tensor is of rank two \cite{BGG18, BGG16}. However linear equivalence is strictly finer than affine equivalence if $\operatorname{det}\rho=0$.}
%We refer to \cite{GV18} for the proof of the following result which 
%describes the filtering of $\mathcal{A}^0$
%by the action of $\operatorname{Gl}(2,\mathbb{R})$.
%
%\begin{theorem}\label{T1.4}
%If $\mathcal{M}$ is a flat Type~$\mathcal{A}$ model, then $\mathcal{M}$ is linearly equivalent to one
%  of the following models:
%  $$
%  \begin{array}{ll}
%  \MM_0^0:=\mathcal{M}(0,0,0,0,0,0),&
%  \MM_1^0:=\mathcal{M}(1,0,0,1,0,0),\\
%  \noalign{\medskip}
%  \MM_2^0:=\mathcal{M}(-1,0,0,0,0,1), &
%  \MM_3^0:=\mathcal{M}(0,0,0,0,0,1),\\
%  \noalign{\medskip}
%  \MM_4^0:=\mathcal{M}(0,0,0,0,1,0), &
%  \MM_5^0:=\mathcal{M}(1,0,0,1,-1,0).
%  \end{array}
%  $$
%\end{theorem}

\begin{proof}[The proof of Theorem~\ref{T1.1}]
Let $\theta\in[0,2\pi]$ be the usual periodic parameter where we identify $0$ with $2\pi$ to define
the circle $S^1=(\cos\theta,\sin\theta)$. 
Let $(x^1,x^2,x^3)$ be a point of $\mathbb{R}^3$. The bundle $\mathbb{L}\oplus\pone\oplus\pone$ is
then defined by identifying $(\theta,x^1,x^2,x^3)$ with $(\theta+\pi,-x^1,x^2,x^3)$; this puts the necessary
half twist in the first $x$-coordinate. We require that $(x^1,x^2,x^3)$ belongs to $\mathbb{R}^3-\{0\}$ to remove the 0-section.

The parametrization of Equation~(\ref{E1.a}) is not a very convenient one for studying the Ricci tensor.
We make a linear change of coordinates on $\mathbb{R}^6$ and let 
$\MM_1(p,q,t,s,v,w)$ be defined by
$$\begin{array}{lll}
\Gamma_{11}{}^1=2 q,&\Gamma_{11}{}^2= p + t,&\Gamma_{12}{}^1=\Gamma_{21}{}^1=w,\\
\Gamma_{12}{}^2=\Gamma_{21}{}^2= q + s,&\Gamma_{22}{}^1=v,&\Gamma_{22}{}^2= p - t\,.
\end{array}$$
We substitute these values in Equation~(\ref{E1.b}) to obtain
$$
\rho=\left(\begin{array}{cc}
p^2 + q^2 - s^2 - t^2 - p w - t w&-(p + t) v + (q + s) w\\
-(p + t) v + (q + s) w&q v - s v + (p - t - w) w\end{array}\right)\,.
$$
We set $\rho=0$. If $v^2+w^2\ne0$, we obtain
\begin{equation}\label{E2.a}\begin{array}{l}
p=(v^2+w^2)^{-1}\{2 s v w+t(w^2-v^2)+w^3\},\text{ and}\\
q=(v^2+w^2)^{-1}\{s(v^2-w^2)+v w (2 t+w)\}\,.
\end{array}\end{equation}
If $v^2+w^2=0$, we obtain a single equation
\begin{equation}\label{E2.b}
p^2+q^2-s^2-t^2=0\,.
\end{equation}
We introduce polar coordinates $v=r\cos(\theta)$ and $w=r\sin(\theta)$ to remove the singularity at $(v,w)=(0,0)$ in
Equation~(\ref{E2.a}). We may then
combine Equation~(\ref{E2.a}) and Equation~(\ref{E2.b}) into a single expression:
\begin{equation}\label{E2.c}\begin{array}{l}
p=p(\theta,r,s,t):=r \sin ^3(\theta )+s \sin (2 \theta )-t \cos (2 \theta ),\\[0.05in]
q=q(\theta,r,s,t):=r \cos (\theta) \sin^2 (\theta)+s \cos (2 \theta )+t \sin (2 \theta )\,.
\end{array}\end{equation}
We assume $(r,s,t)\ne (0,0,0)$ to avoid the trivial { structure $\mathcal{M}_0^0$}
as the parametrization of Equation~(\ref{E2.c}) is singular there.
We have $\theta\in[0,2\pi]$ and $(r,s,t)\in\mathbb{R}^3-\{0\}$; since we are permitting $r$ to
be negative in polar coordinates, we must identify $(\theta,r)$ with $(\theta+\pi,-r)$ and obtain thereby the bundle
$\mathbb{L}\oplus\pone\oplus\pone$ minus the zero section over $[0,\pi]$. 
\end{proof}

\begin{remark}\rm
  The isotropy subgroups of the { structures} $\mathcal{M}_i^0$ { vary} with $i$
  { and the dimension of the orbit space varies correspondingly}. 
  We list below  the associated isotropy subgroups.
   $$
  \begin{array}{l}
  \mathcal{I}(\MM_0^0)=\operatorname{GL}(2,\mathbb{R}),\\
  \noalign{\medskip}
  \mathcal{I}(\MM_1^0)=\left\{T:T(x^1,x^2)=(x^1,a x^2) \text{ for } a\ne0\right\},\\
  \noalign{\medskip}
  \mathcal{I}(\MM_2^0)
  =\left\{id, T\right\}, \text{ where } T(x^1,x^2)=(-x^2,-x^1), \\
  \noalign{\medskip}
  \mathcal{I}(\MM_3^0)=\left\{T:T(x^1,x^2)=(a x^1,x^2)\text{ for } a\ne0\right\},
  \\
  \noalign{\medskip}
  \mathcal{I}(\MM_4^0)=\left\{T:T(x^1,x^2)=(a^2 x^1+bx^2,ax^2)\text{ for }a\ne0,\ b\in\mathbb{R}\right\}, \\
  \noalign{\medskip}
  \mathcal{I}(\MM_5^0)=\left\{T:T(x^1,x^2)=(x^1,\pm x^2)\right\}. 
  \end{array}
  $$
%  $$
%  \begin{array}{ll}
%  \mathcal{I}(\MM_0^0)=\operatorname{GL}(2,\mathbb{R}),&\quad
%  \mathcal{I}(\MM_1^0)=\left\{\left(\begin{array}{cc}1&0\\0&a\end{array}\right):\ a\ne0\right\},
%  \\
%  \noalign{\medskip}
%  \mathcal{I}(\MM_2^0)
%  =\left\{\left(\begin{array}{cc}1&0\\0&1\end{array}\right),\left(\begin{array}{cc}0&-1\\-1&0\end{array}\right)\right\},&\quad
%  \mathcal{I}(\MM_3^0)=\left\{\left(\begin{array}{cc}a&0\\0&1\end{array}\right):\ a\ne0\right\},
%  \\
%  \noalign{\medskip}
%  \mathcal{I}(\MM_4^0)=\left\{\left(\begin{array}{cc}a^2&b\\0&a\end{array}\right):\ a\ne0,\ b\in\mathbb{R}\right\},&\quad
%  \mathcal{I}(\MM_5^0)=\left\{\left(\begin{array}{cc}1 & 0 \\0 & \pm 1 \\\end{array}\right)\right\}. 
%  \end{array}
%  $$
\end{remark}

\subsection{The space of Type~$\mathcal{A}$ models with rank-one Ricci tensor}
If the Ricci tensor has rank 1, we can make a linear
change of coordinates to ensure $\rho$ is a multiple of $dx^2\otimes dx^2$. { We first
establish Theorem~\ref{T1.2}. We then examine the isotropy groups of the models in Equation \eqref{eq:models1}
 to determine the orbits of the Type~$\mathcal{A}$ models which are not Type~$\mathcal{B}$.}

\begin{lemma}\label{L3.1}
 Let $\mathcal{M}$ be a Type~$\mathcal{A}$ model which is not flat.
 Then $\rho$ is a multiple of $dx^2\otimes dx^2$ if and only if $b=0$ and $d=0$.
\end{lemma}
\begin{proof}
 A direct computation shows 
 \begin{equation}\label{E1.b}
 \rho=\left(\begin{array}{cc}
 (a-d) d+b (f-c) & c d-b e \\
 c d-b e & c(f-c)+(a-d) e \\
 \end{array}\right)\,.
 \end{equation} 
 Consequently, if $b=0$ and if $d=0$, then $\rho$ is a multiple of $dx^2\otimes dx^2$. Conversely, assume $\rho$ is a
 multiple of $dx^2\otimes dx^2$ or, equivalently, $-bc+ad-d^2+bf=0$ and $cd-be=0$.  We wish to show $b=d=0$.
 \smallbreak{\bf Case 1.} Suppose that $d\ne0$. The equations
 are homogeneous so we may assume $d=1$ and hence $c=be$. Substituting these values yields $\rho_{11}=-1+a-b^2e+bf=0$.
 Thus $a=1+b^2e+bf$. This yields $\rho=0$ so this case is impossible as we assumed $\MM$ was not flat.
 \smallbreak{\bf Case 2.} Suppose that $b\ne0$. Again, we may assume $b=1$ so $e=cd$. We compute 
 $\rho_{11}=f-c+ad-d^2$. Setting this to zero again yields $\rho=0$ which is impossible. 
\end{proof}
 
%We refer to \cite{GV18} for the proof of the following result which 
%describes the filtering of  $\mathcal{A}^1_\pm$
%by the action of $\operatorname{Gl}(2,\mathbb{R})$.
%
%\begin{theorem}\label{T1.4b}
%If $\mathcal{M}$ is a Type~$\mathcal{A}$ model with $\operatorname{Rank}\{\rho\}=1$,
%  then $\mathcal{M}$ is linearly equivalent to one
%  of the following models:
%  $$\red
%  \begin{array}{ll}
%  \MM_1^1:=\mathcal{M}(2,0,0,1,0,-1),&\quad \rho=dx^1\otimes dx^1,\\
%  \noalign{\medskip}
%  \MM_2^1(c_1):=\mathcal{M}(1+2c_1,0,0,c_1,0,-1),&\quad \rho=c(1+c)dx^1\otimes dx^1,\,\,(c_1\ne0,-1),\\
%  \noalign{\medskip}
%      \MM_3^1(c_1):=\mathcal{M}(1+2c_1,0,0,c_1,0,0), &\quad \rho=c(1+c)dx^1\otimes dx^1,\,\, (c_1\ne0,-1),\\
%  \noalign{\medskip}
%  \MM_4^1(c):=\mathcal{M}(2,c,0,1,0,0), &\quad\rho=dx^1\otimes dx^1,\\
%  \noalign{\medskip}
%  \MM_5^1(c):=\mathcal{M}(2c,1+c^2,0,0,0,1), &\quad\rho=(1+c^2)dx^1\otimes dx^1.
%  \end{array}
%  $$
%\end{theorem}

\begin{proof}[Proof of Theorem~\ref{T1.2}]
 Let $\mathcal{A}^{1}_{\pm,0}$ be the space of all Type~$\mathcal{A}$ models where the Ricci tensor is a non-zero multiple of 
 $dx^2\otimes dx^2$
 where the $\pm$ refers to whether $\rho_{22}$ is positive or negative.
 By Lemma~\ref{L3.1}, we set $b=d=0$ and obtain $\rho_{22}=-c^2+ae+cf$.
 We make a change of variables setting
 $$\begin{array}{llllll}
 a=q+v,&b=0,&c=u+p,&
 d=0,&e=q-v,&f=2p.
 \end{array}$$
 We then have $\rho_{22}=(p^2+q^2-u^2-v^2)dx^2\otimes dx^2$ so we may identify
 \begin{eqnarray*}
  &&\mathcal{A}_{+,0}^1=\{\Gamma(p,q,u,v):p^2+q^2>u^2+v^2\},\\
  &&\mathcal{A}_{-,0}^1=\{\Gamma(p,q,u,v):p^2+q^2<u^2+v^2\}\,.
 \end{eqnarray*}
 We examine $\mathcal{A}_{-,0}^1$ as the analysis of 
 $\mathcal{A}_{+,0}^1$ is the same after interchanging the roles of $(p,q)$ and $(u,v)$.
 Let $\mathcal{D}^2:=\{(U,V)\in\mathbb{R}^2:U^2+V^2<1\}$ be the open
 disk in $\mathbb{R}^2$.  Let $-\MM$ be the Type~$\mathcal{A}$ model $\mathcal{M}(-a,-b,-c,-d,-e,-f)$. 
 We construct a diffeomorphism
 $\Phi$ from $S^1\times\mathbb{R}^+\times\mathcal{D}^2$ to $\mathcal{A}_{-,0}^1$ by setting
 $u=r\cos\theta$, $v=r\sin\theta$, $p=rU$, $q=rV$. For $r>0$, $\theta\in S^1$, and $U^2+V^2<1$ we have
 $$
 \mathcal{M}=\mathcal{M}(r (\sin (\theta )+V),0,r (\cos (\theta )+U),0,r (V-\sin (\theta )),2 rU)\,.
 $$
 It is clear that $-\MM(\theta,r,U,V)=\MM(\theta+\pi,r,-U,-V)$.
 
 Let $\tilde\MM$ be an arbitrary Type~$\mathcal{A}$ model with $\operatorname{Rank}\{\rho_{\tilde\MM}\}=1$ and 
 $\rho_{\tilde\MM}$ negative semi-definite. We may express
 $$\rho_{\tilde\MM}=\lambda(\cos(\phi)dx^2-\sin(\phi)dx^1)\otimes(\cos(\phi)dx^2-\sin(\phi)dx^1)
 $$
for $\lambda<0$. Here $\phi$ is only defined modulo $\pi$ instead of the usual $2\pi$.
Let
 $$
 T_\phi(x^1,x^2)=(\cos(\phi)x^1+\sin(\phi)x^2,-\sin(\phi)x^1+\cos(\phi)x^2)\,.
 $$
 be the associated rotation so that
 $T_\phi^*(dx^2)=-\sin(\phi)dx^1+\cos(\phi) dx^2$ and thus $(T_\phi)_*\tilde{\mathcal{M}}$
 belongs to $\mathcal{A}_{-,0}^1$. We then have
 $$
 \mathcal{A}_-^1=\{\mathbb{R}/(2\pi\mathbb{Z})\times\mathcal{A}_{-,0}^{1}\}/(\phi,\MM)\sim(\phi+\pi,-\MM)
 $$
 where the gluing reflects the fact that when $\phi=\pi$ we have replaced $(x^1,x^2)$ by $(-x^1,-x^2)$ and
 thus changed the sign of the Christoffel symbols. Using our previous parametrization of $\mathcal{A}_{-,0}^{1}$, this yields
 $$
 \mathcal{A}_-^1=(\mathbb{R}^2/(2\pi\mathbb{Z})^2)\times\mathbb{R}^+\times\mathcal{D}^2/
 \{(\phi,\theta,r,U,V)\sim(\phi+\pi,\theta+\pi,r,-U,-V)\}\,.
 $$
 After setting $\tilde\theta=\theta+\phi$, we can rewrite this equivalence relation in the form
 $$
 (\phi,\tilde\theta,r,U,V)\sim(\phi+\pi,\tilde\theta,r,-U,-V)\,.
 $$
 The variable $\tilde\theta$ now no longer plays a role in the gluing. After replacing $\mathbb{R}^+$ by $\mathbb{R}$ and $\mathcal{D}^2$ by $\mathbb{R}^2$,
 we see $\mathcal{A}_-^1$ is diffeomorphic to $S^1\times S^1\times\mathbb{R}^3$ modulo the relation
 $$
 (\phi,\tilde\theta,x_1,x_2,x_3)\sim(\phi+\pi,\tilde\theta,x_1,-x_2,-x_3)\,.
 $$
 These gluing relations define the total space of
 the bundle $\pone\oplus\mathbb{L}\oplus\mathbb{L}$ over $(S^1,\phi)$. Since $\mathbb{L}\oplus\mathbb{L}$ is 
 diffeomorphic to the trivial 2-plane bundle
 $\pone\oplus\pone$, we obtain finally that $\mathcal{A}_{-}^1$ is diffeomorphic to $S^1\times S^1\times\mathbb{R}^3$.
\end{proof}
 
We adopt the notation of { Equation~\eqref{eq:models1}} to describe the orbits of the
models $\mathcal{M}_i^1(\cdot)$ in the following lemma. 
\begin{lemma}\label{L4.2}
 \ \begin{enumerate}
  \item $\mathcal{I}(\mathcal{M}_1^1)=\{\operatorname{id}\}$.
  \item $\mathcal{I}(\mathcal{M}_2^1(c_1))=\{\operatorname{id}\}$ if $c_1\ne-\frac12$.
  \item $\mathcal{I}(\mathcal{M}_2^1(-\frac{1}{2}))=\{\operatorname{id},T\}$, where $T(x^1,x^2)=(x^1+x^2,-x^2)$.
  \item $\mathcal{I}(\mathcal{M}_3^1(c_1))=\{T:T(x^1,x^2)=(v^{-1}x^1,x^2) \text{ for } v\in\mathbb{R}\backslash \{0\}\}$.
  \item $\mathcal{I}(\mathcal{M}_4^1(c))=\{T:T(x^1,x^2)=(x^1-wx^2,x^2)\text{ for }w\in\mathbb{R}\}$,  if $c\ne0$, .
  \item $\mathcal{I}(\mathcal{M}_4^1(0))=\{T:T(x^1,x^2)=(v^{-1}(x^1-wx^2),x^2)\text{ for }w\in\mathbb{R}, v\in\mathbb{R}\backslash \{0\}\}$.
  \item $\mathcal{I}(\mathcal{M}_5^1(c))=\{\operatorname{id}\}$, if $c\ne0$.
  \item $\mathcal{I}(\mathcal{M}_5^1(0))=\{\operatorname{id},T\}$ where $T(x^1,x^2)=(x^1,-x^2)$.
\end{enumerate}\end{lemma}
\begin{proof} Suppose $T\in\mathcal{I}(\mathcal{M}_i^1(\cdot))$. The
 Ricci tensor of $\mathcal{M}_i^1(\cdot)$ is a non-zero multiple of $dx^2\otimes dx^2$. Since $T$ must preserve the
 Ricci tensor, $T(dx^2)=\pm dx^2$. This implies
  $(y^1,y^2)=T(x^1,x^2)=(v^{-1}(x^1-wx^2),\varepsilon x^2)$ for $\varepsilon=\pm1$. Then
  \par $dy^1=v^{-1}(dx^1-wdx^2)$, \quad $dy^2=\varepsilon dx^2$,\quad 
  $\partial_{y^1}=v\partial_{x^1}$,\quad$\partial_{y^2}=\varepsilon(w \partial_{x^1}+\partial_{x^2})$,
  \par${}^y\Gamma_{11}{}^1 := v ({}^x\Gamma_{11}{}^1 - w\ {}^x\Gamma_{11}{}^2)$.
  \par${}^y\Gamma_{11}{}^2 := v^2\varepsilon {}^x\Gamma_{11}{}^2$,
  \par${}^y\Gamma_{12}{}^1 := \varepsilon ({}^x\Gamma_{12}{}^1 +
  w\ ( {}^x\Gamma_{11}{}^1 - {}^x\Gamma_{12}{}^2 - w\ {}^x\Gamma_{11}{}^2))$,
  \par${}^y\Gamma_{12}{}^2 := v ({}^x\Gamma_{12}{}^2 + w\  {}^x\Gamma_{11}{}^2)$,
  \par${}^y\Gamma_{22}{}^1 := \frac{1}v({}^x\Gamma_{22}{}^1 + w (2\ {}^x\Gamma_{12}{}^1
 - {}^x\Gamma_{22}{}^2) + w^2 ({}^x\Gamma_{11}{}^1 - 2\ {}^x\Gamma_{12}{}^2) - w^3\ {}^x\Gamma_{11}{}^2)$,
  \par${}^y\Gamma_{22}{}^2 := \varepsilon({}^x\Gamma_{22}{}^2 + 2 w\ {}^x\Gamma_{12}{}^2 + w^2\ {}^x\Gamma_{11}{}^2)$.
 
 \medbreak\noindent{\bf Case 1.}
 $\mathcal{M}_1^1=\mathcal{M}(-1,0,1,0,0,2)$ and $T^*\mathcal{M}_1^1=\mathcal{M}(-v,0,\varepsilon(1-w),0,-\frac{w^2}v,2\varepsilon)$.
 Examining $\Gamma_{11}{}^1$ and $\Gamma_{22}{}^2$ yields $\varepsilon=1$ and $v=1$. Examining $\Gamma_{22}{}^1$ yields $w=0$.
 \medbreak\noindent{\bf Case 2.} We have $c\notin\{0,-1\}$,
 $\mathcal{M}_2^1(c)=\mathcal{M}(-1,0,c,0,0,1+2c)$, and
$$\textstyle
{ T^*\mathcal{M}_2^1(c)=
\mathcal{M}(-v,0,\varepsilon(c-w),0,-\frac{1}v(w+w^2),(1+2c)\varepsilon)\,.}
$$
 Examining $\Gamma_{11}{}^1$
 yields $v=1$. Suppose $c\ne-\frac12$. Examining $\Gamma_{22}{}^2$
 yields $\varepsilon=1$. Since $\varepsilon=1$, examining $\Gamma_{12}{}^1$  yields $w=0$. 
 Suppose $c=-\frac12$. Examining {$\Gamma_{12}{}^1$ and $\Gamma_{22}{}^1$} yields
 $(\varepsilon,w)=(1,0)$ or $(\varepsilon,w)=(-1,-1)$.
 \smallbreak\noindent{\bf Case 3.} We have $c\notin\{0,-1\}$,
 $\mathcal{M}_3^1(c)=\mathcal{M}(0,0,c,0,0,1+2c)$, and
 $$\textstyle
 { T^*\mathcal{M}_3^1(c)=\mathcal{M}(0,0,c\varepsilon,0,-\frac{w}v,(1+2c)\varepsilon)}\,.
 $$
 Examining $\Gamma_{12}{}^1$ yields $\varepsilon=1$.
 Examining $\Gamma_{22}{}^1$ yields $w=0$. There  is then no condition on $v$.
 \smallbreak\noindent{\bf Case 4.}
 $\mathcal{M}_4^1(c)=\mathcal{M}(0,0,1,0,c,2)$ and $T^*\mathcal{M}_4^1(c)=\mathcal{M}(0,0,\varepsilon,0,\frac{c}v,2\varepsilon)$.
 Examining $\Gamma_{22}{}^2$ yields $\varepsilon=1$. There is no condition on $w$. If $c\ne0$, examining $\Gamma_{22}{}^1$ yields $v=1$;
 if $c=0$, there is no condition on $v$.
 \smallbreak\noindent{\bf Case 5.}
 $\mathcal{M}_5^1(c)=\mathcal{M}(1,0,0,0,1+c^2,2c)$ and
$$
\textstyle{ T^*\mathcal{M}_5^1(c)=\mathcal{M}(v,0,w\varepsilon,0,\frac{1}v(1+(c-w)^2),
2c\varepsilon)\,.}
$$
 Examining $\Gamma_{11}{}^1$ shows $v=1$.
 Examining $\Gamma_{12}{}^1$ shows $w=0$.
 If $c\ne0$, examining $\Gamma_{22}{}^2$ shows $\varepsilon=1$. If $c=0$, we obtain $\varepsilon=\pm1$.
\end{proof}

The general linear group $\operatorname{GL}(2,\mathbb{R})$ acts on the space $\mathbb{R}^6$ of all Type~$\mathcal{A}$ geometries 
via change of coordinates. Let $\operatorname{GL}_+(2,\mathbb{R})$
be the subgroup of matrices with positive determinant. 
If $\mathcal{M}$ is a Type~$\mathcal{A}$ model with $\operatorname{Rank}\{\rho\}(\mathcal{M})=2$, then
the associated space of affine Killing vector fields is 2-dimensional and $\mathcal{M}$ does not also admit
a Type~$\mathcal{B}$ structure \cite{BGG18}. But there are Type~$\mathcal{A}$ models with 
$\operatorname{Rank}\{\rho\}=1$ which also admit Type~$\mathcal{B}$ structures. 
Let $\mathcal{O}^1_{\pm}\subset\mathcal{A}^1_{\pm}$ be the set of
Type~$\mathcal{A}$
models with $\operatorname{Rank}\{\rho\}=1$ and which do not admit Type~$\mathcal{B}$ structures.

\begin{theorem}\label{T1.3}%%?? Start here
 \ \begin{enumerate}
  \item $\mathcal{O}^1_{-}$ is empty; every element of $\mathcal{A}^1_{-}$ also admits Type~$\mathcal{B}$ structure.
  \item
  $\operatorname{GL}_+(2,\mathbb{R})$ acts without fixed points on
  $\mathcal{O}^1_{+}$. The action admits a section $s:\mathbb{R}\rightarrow\mathcal{O}^1_{+}$ so
  $\mathcal{O}^1_{+}=\operatorname{GL}_+(2,\mathbb{R})\times\mathbb{R}$ is a principal fiber bundle over $\mathbb{R}$.
\end{enumerate}
\end{theorem}
\begin{proof}
Resuts of \cite{BGG18} show that the models $\mathcal{M}_i^1(\cdot)$ for $1\le i\le 4$ 
also admit Type~$\mathcal{B}$ structures while
the models $\mathcal{M}_5^1(c)$ do not. 
 The Ricci tensor associated to $\mathcal{M}_i^1(\cdot)$ is given by:
   $${
   \begin{array}{lll}
    \rho^{\MM_1^1}= dx^2\otimes dx^2,& \,\,
    \rho^{\MM_2^1}=c_1(1+c_1) dx^2\otimes dx^2,\\[0.05in]
    \rho^{\MM_3^1}=c_1(1+c_1) dx^2\otimes dx^2,& \,\,
\rho^{\MM_4^1}= dx^2\otimes dx^2,\\[0.05in]
   \rho^{\MM_5^1}=(1+c^2) dx^2\otimes dx^2.
   \end{array}}
   $$
If $\rho\le0$, then it follows that
$i=2$ or $i=3$ and $c\in(-1,0)$. Thus any element of $\mathcal{A}^1_{-}$ admits a Type~$\mathcal{B}$ structure which
proves Assertion~(1).

Let $\mathfrak{M}_5^1=\cup_c\mathcal{M}_5^1(c)$; this is a smooth curve in $\mathbb{R}^6$.
Type $\mathcal{A}$ models which are linearly equivalent to $\mathcal{M}_1^1$, $\mathcal{M}_2^1(c_1)$ for $c_1+c_1^2>0$,
$\mathcal{M}_3^1(c_1)$ for $c_1+c_1^2>0$, or $\mathcal{M}_4^1(c)$ all admit Type~$\mathcal{B}$ structures and have $\rho\ge0$.
Thus we may identify the structures $\mathcal{O}_{1,+}$ which do not admit Type~$\mathcal{B}$
structures with $\operatorname{GL}(2,\mathbb{R})\cdot\mathfrak{M}_5^1$. 
Let $T(x^1, x^2):=(x^1, -x^2)$. We have $T\mathcal{M}_5^1(c)=\mathcal{M}_5^1(-c)$.
Since $\det(T)=-1$, we conclude therefore that $\mathcal{O}^1_{+}=\operatorname{GL}_+(2,\mathbb{R})\cdot\mathfrak{M}_5^1$.
By Lemma~\ref{L4.2}, the action of
$\operatorname{GL}_+(2,\mathbb{R})$ on $\mathfrak{M}_5^1$ is fixed point free. Assertion~(2) follows.
\end{proof}

\section{The space of Type~$\mathcal{B}$ connections}\label{section-type-B}
Let 
$\mathcal{N}(a,b,c,d,e,f):=(\mathbb{R}^+\times\mathbb{R},\nabla)$ where the Christoffel symbols of $\nabla$ are given by
\eqref{E1.c}. The Ricci tensor needs not be symmetric in this setting:
\begin{equation}\label{E1.d}\begin{array}{l}
\rho=(x^1)^{-2}\left(
\begin{array}{cc}
(a-d+1) d+b (f-c) & c d-b e+f \\
c (d-1)-b e & -c^2+f c+(a-d-1) e \\
\end{array}
\right)
\end{array}\end{equation}

\subsection{The space of flat Type~$\mathcal{B}$ models}

\begin{proof}[The proof of Theorem~\ref{T1.5}]
Let $\mathcal{N}=\mathcal{N}(a,b,c,d,e,f)$. We clear denominators
in Equation~\eqref{E1.d} and set $\tilde\rho_{ij}=(x^1)^2\rho_{ij}$.
Adopt the notation of Equation \eqref{eq:paramet1}.
A direct computation shows the structures
$\mathcal{U}_i(\cdot)$ are flat. We distinguish cases to establish the converse. We use
Equation~(\ref{E1.d}) and set $\tilde\rho=0$. Since
$\tilde\rho_{12}-\tilde\rho_{21}=c+f$, $f=-c$.
\smallbreak\noindent{\bf Case 1.} Assume $e\ne0$. Set $c=rs$, $e=r$, and $f=-rs$ for $r\ne0$.
Then
$$
\tilde\rho_{22}=-r (1 - a + d + 2 r s^2 )\text{ and }
\tilde\rho_{21}= -r(b+s-ds)\,.
$$
We solve these equations to obtain $a = 1 + d + 2 r s^2 $ and $b=(-1 + d) s $. We have
$\tilde\rho_{11}=2(d+ rs^2)$.
Thus $d=-rs^2$ which gives the parametrization $\mathcal{U}_1$.
\smallbreak\noindent{\bf Case 2.} Suppose $e=0$. Set $a=u$, $b=v$, and $f=-c$ to obtain
$$
\tilde\rho=\left(
\begin{array}{cc}
 d(1+u-d)-2 c v & c (d-1) \\
 c (d-1) & -2  c^2
\end{array}
\right)\,.
$$
This yields $c=0$ and $d(1+u-d)=0$.
If we set $d=0$, we obtain the parametrization $\mathcal{U}_2$; if we set $d=1+u$, we obtain
the parametrization $\mathcal{U}_3$. This establishes the first assertion.

The parametrization $\mathcal{U}_2$ and $\mathcal{U}_3$ intersect when $u=-1$; 
the intersection is transversal along
the curve $\mathcal{N}( -1,v,0,0,0,0)$. We wish to extend the parametrization $\mathcal{U}_1$
to study the limiting behavior as $e\rightarrow0$. We distinguish cases.
\smallbreak\noindent{\bf Case A.} Suppose $\lim_{n\rightarrow\infty}\mathcal{U}_1(r_n,s_n)\in\operatorname{Range}\{\mathcal{U}_2\}$.
We have 
$$\begin{array}{lll}
\lim_{n\rightarrow\infty}1+r_ns_n^2=u,&\lim_{n\rightarrow\infty}-s_n(1+r_ns_n^2)=v,&\lim_{n\rightarrow\infty}-r_ns_n=0,\\
\lim_{n\rightarrow\infty}-r_ns_n^2=0,&\lim_{n\rightarrow\infty}r_n=0,&\lim_{n\rightarrow\infty}-r_ns_n=0.
\end{array}$$
These equations imply $u=1$, $\lim_{n\rightarrow\infty}r_n=0$, $\lim_{n\rightarrow\infty}s_n=-v$. Thus we may simply set $r=0$
to obtain a transversal intersection along the curve $\mathcal{N}(1,v,0,0,0,0)$.

\smallbreak\noindent{\bf Case B.} Suppose $\lim_{n\rightarrow\infty}\mathcal{U}_1(r_n,s_n)\in\operatorname{Range}\{\mathcal{U}_3\}$.
We have
$$\begin{array}{lll}
\lim_{n\rightarrow\infty}1+r_ns_n^2=u,&\lim_{n\rightarrow\infty}-s_n(1+r_ns_n^2)=v,&\lim_{n\rightarrow\infty}-r_ns_n=0,\\
\lim_{n\rightarrow\infty}-r_ns_n^2=1+u,&\lim_{n\rightarrow\infty}r_n=0,&\lim_{n\rightarrow\infty}-r_ns_n=0.
\end{array}$$
These equations imply $u=0$, $\lim_{n\rightarrow\infty} r_n=0$, and $\lim_{n\rightarrow\infty}r_ns_n^2=-1$.
We change variables setting $r=-t^2$ and $s=\frac1t+w$ to express
$$\begin{array}{llll}
\mathcal{U}_1 (-t^2,\frac1t+w)=\mathcal{N} (&- t w(2+t w), & w(2 + 3 t w + t^2 w^2), &-t(1+ t w), \\[0.05in]
&(1 + t w)^2 , & -t^2,&  t(1 + t w))\,.
\end{array}$$
We may now safely set $t=0$ to obtain the intersection with $\operatorname{Range}\{\mathcal{U}_3\}$ along the curve
$\mathcal{N}(0,2w,0,1,0,0)$.
\end{proof}

\subsection{Type~$\mathcal{B}$ models with alternating Ricci tensor}

It was shown in \cite{BGG18} that any Type~$\mathcal{B}$ model with alternating Ricci tensor is linearly equivalent to one of the following models:
 $$
 \begin{array}{ll}
 \mathcal{N}_1(c):=\mathcal{N}(0,c,1,0,0,1), & \text{for}\,\, c\in\mathbb{R},
 \\
 \noalign{\medskip}
 \mathcal{N}_2(c,\pm):=\mathcal{N}(1\mp c^2),c,0,\mp c^2,\pm 1,\pm 2c), &\text{for}\,\, c>0.
 \end{array}
 $$

\begin{proof}[The proof of Theorem~\ref{T1.7}]
Adopt the notation of Equation \eqref{eq:paramet2}. It is clear that $\mathcal{V}_1$ defines a smooth 3-dimensional submanifold of $\mathbb{R}^6$. To see similarly
that $\mathcal{V}_2$ is smooth, we note that we can recover
$u=\frac12( c+f)$ and $v=e$.  If $v\ne0$, then $w=\frac1v(f-u)$
while if $v=0$, $w= \frac1{2u}(1-a)$. Thus $\mathcal{V}_2$ is 1-1; it is 
not difficult to verify the Jacobian determinant
is non-zero. This shows that $\mathcal{V}_2$ also defines a smooth 3-dimensional submanifold
of $\mathbb{R}^6$. We set $v=0$ and $u=r$ to see that $\mathcal{V}_1$ and $\mathcal{V}_2$
intersect along the surface $v=0$, $u=r$, $s=1-2 uw$ and $t=w(1-uw)$. A direct computation shows
the associated Ricci tensors are non-trivial and alternating:
$$
 \tilde\rho_{\mathcal{V}_1}=r\left(\begin{array}{cc}0&1\\-1&0\end{array}\right)\text{ and }
 \tilde\rho_{\mathcal{V}_2}=u\left(\begin{array}{cc}0&1\\-1&0\end{array}\right)\,.
$$
Let $\mathcal{N}$ be a Type~$\mathcal{B}$ model with $\rho_s=0$ and 
$\tilde\rho_{a,12}= \frac{c+f}2\ne0$. We distinguish cases.
\smallbreak\noindent{\bf Case 1.} Suppose $e=0$. Set $c=2r-f$ for $r\ne0$.
Setting the $\rho_s=0$ yields
$$\begin{array}{ll}
\rho_{s,11}:\ 0= d(1+a-d) + 2 b (f - r),&
\rho_{s,12}:\ 0=(1-d ) f + r( 2 d-1),\\
\rho_{s,22}:\ 0=-2 (f^2 - 3 f r + 2 r^2).
\end{array}$$
We solve the equation $-2 (f^2 - 3 f r + 2 r^2)=0$ to obtain 
$f=r$ or $f=2r$. Setting $f=2r$ yields $\rho_{s12}$: $0=r$ which is false.
Thus $f=r$. We obtain $\rho_{s,12}=2dr$ so $d=0$. Set
$a=s$ and $b=t$ to obtain the parametrization $\mathcal{V}_1$.
\smallbreak\noindent{\bf Case 2.} Set $c=2u-f$ and $e=v$ for $u\ne0$ and $v\ne0$. We obtain
\begin{eqnarray*}
&&\rho_{s,11}:\ 0= d(1+a-d) + 2 b( f - u), \\
&&\rho_{s,12}:\ 0=  (1 - d) f - u + 2 d u - b v,\\
&&\rho_{s,22}:\ 0=  -2 f^2 + 6 f u - 4 u^2 - (1 - a + d) v.
\end{eqnarray*}
Setting $\rho_{s,12}=0$ and $\rho_{s,22}=0$ yields
$\textstyle a=\frac1v(2 f^2 - 6 f u + 4 u^2 + v + d v)$ and
$\textstyle b =\frac1v(f - d f - u + 2 d u)$.
We obtain
$\rho_{s,11}=\frac1v(2 (f^2 - 2 f u + u^2 + d v))$. 
This implies that $d=\textstyle-\frac{(f-u)^2}v$. 
Setting $f=vw+u$ yields the parametrization $\mathcal{V}_2$.
This parametrization can be extended safely to $v=0$; we require $u\ne0$ to ensure $\rho_a\ne0$.
\end{proof}

\end{document}